\documentclass[10pt,journal]{IEEEtran}
\usepackage{amsfonts}
\usepackage{amsmath,amssymb}
\usepackage{graphicx}
\usepackage{caption2}
\usepackage{epsfig}
\usepackage[dvips]{color}

\def\dref#1{(\ref{#1})}
\def\rm{\mathrm}
\newtheorem{theorem}{Theorem}
\newtheorem{lemma}{Lemma}

\newtheorem{remark}{Remark}

\begin{document}

\title{Fully Distributed Adaptive Controllers for Cooperative Output Regulation
of Heterogeneous Linear Multi-agent Systems with Directed Graphs}
%
%
\author{Zhongkui~Li, Michael Z.Q. Chen, and Zhengtao~Ding%
\thanks{This work was supported by the National Natural Science Foundation
of China under grants 61473005 and 11332001, and
a Foundation for the Author of National Excellent Doctoral Dissertation of PR China.}
\thanks{Z. Li is with the State Key Laboratory for Turbulence
and Complex Systems,
Department of Mechanics and Engineering Science, College of Engineering, Peking University, Beijing 100871, China
(E-mail: zhongkli@pku.edu.cn).}
\thanks{M.Z.Q. Chen is with the Department of Mechanical Engineering,
The University of Hong Kong, Pokfulam, Hong Kong (E-mail: mzqchen@hku.hk).}
\thanks{Z. Ding is with the Control Systems Centre, School of Electrical and Electronic Engineering, University of Manchester, Sackville Street Building, Manchester M13 9PL, UK (E-mail: zhengtao.ding@manchester.ac.uk).}}


\maketitle

\begin{abstract}
This paper considers the cooperative output regulation problem for linear multi-agent systems with a directed communication graph, heterogeneous linear subsystems, and an exosystem whose output is available to only a subset of subsystems.
Both the cases with nominal and uncertain linear subsystems are studied.
For the case with nominal linear subsystems, a distributed adaptive observer-based controller is designed, where the distributed adaptive observer is implemented for the subsystems to estimate the exogenous signal. For the case with uncertain linear subsystems,
the proposed distributed observer and the internal model principle are combined to solve the robust cooperative output regulation problem.
Compared with the existing works, one main contribution of this paper is that the proposed control schemes
can be designed and implemented by each subsystem in a fully distributed fashion for general directed graphs.
For the special case with undirected graphs, a distributed output feedback control law is further presented.
\end{abstract}


\IEEEpeerreviewmaketitle

\section{Introduction}

Cooperative output regulation of multi-agent systems is to have a group of autonomous agents (subsystems)
interacting with each other via communication or sensing
to asymptotically track a prescribed trajectory 
and/or maintain asymptotic rejection of disturbances.
The cooperative output regulation problem is closely related to the consensus problem and other cooperative control
problems as studied in 
\cite{ren2007information,cao2013overview,li2014cooperative,6573374,6560353,5898403,14181306,14351657} and the references therein.
Actually, the cooperative output regulation problem contains the leader-follower consensus
or distributed tracking problem as special cases.
A central work in cooperative output regulation is to design appropriate
distributed controllers, depending
on only the local state or output information of each agent and its neighbors.
Considering the flexibility and reconfigurability
that multi-agent systems are expected to maintain
and meanwhile the limited sensing or communicating capacity
that the agents have,
distributed control schemes, compared with centralized ones,
are believed to be more favorable.

In the recent years, the cooperative output regulation problem has been extensively investigated
by many researchers. Many interesting results are reported, e.g., in \cite{xiang2009synchronized,su2012cooperative,li2013distributeds,meng2013coordinated,su2013general,yu2013robust,wang2010distributed,depersisinternal,isidori2013robust,ding2013consensus,ding2015adaptive}.
In particular, several state and output feedback control laws are proposed in \cite{xiang2009synchronized,su2012cooperative,li2013distributeds,meng2013coordinated}
to achieve cooperative output regulation for multi-agent systems with heterogeneous but known linear subsystems.
The robust cooperative output regulation problem of uncertain linear multi-agent systems is studied in \cite{su2013general,yu2013robust,wang2010distributed},
where internal-model-based controllers are designed.
In \cite{depersisinternal,isidori2013robust,ding2013consensus,ding2015adaptive}, cooperative global output regulation is discussed
for several classes of nonlinear multi-agent systems.

Although many advances have been reported on the cooperative output regulation problem, some challenging issues remain unresolved.
For instance, control design presented in \cite{su2012cooperative,su2013general,yu2013robust,wang2010distributed} explicitly depends on certain nonzero eigenvalues of the Laplacian matrix associated with the communication graph.
However, it is worth mentioning that any nonzero eigenvalue of the Laplacian matrix is global information of the communication graph.
Using these global information of the communication graph 
prevents fully distributed implementation of the controllers.
In other words, the controllers given in the aforementioned papers are not fully distributed.
In \cite{li2013distributeds}, fully distributed
adaptive controllers are proposed, which implement adaptive laws to update the time-varying coupling weights between neighboring agents.
Similar adaptive protocols have been also presented in \cite{li2012adaptiveauto,li2011adaptive,yu2013distributed,li2014TAC} to solve the leaderless and leader-follower consensus problems.
It is worth noting that the adaptive controllers in \cite{li2013distributeds} are applicable to only the case where
the graph among the agents are undirected and that the adaptive protocols in \cite{li2012adaptiveauto,li2011adaptive,yu2013distributed,li2014TAC} are designed for homogeneous multi-agent systems.
To design fully distributed controllers to achieve cooperative output regulation for heterogeneous multi-agent systems with general directed graphs is much more challenging, due to both the heterogeneity of the agents
and the asymmetry of the directed graphs,
and is still open, to the best knowledge of the authors.

This paper extends the fully distributed control design to the cooperative output regulation problem for linear multi-agent systems with a general directed communication graph, heterogeneous linear subsystems, and an exosystem whose output is available to only a subset of subsystems.
Both the cases with nominal and uncertain linear subsystems are studied.
A distributed adaptive observer-based controller is designed to solve the cooperative output regulation problem
for multi-agent systems with nominal linear subsystems.
The distributed adaptive observer, which utilizes the observer states from neighboring subsystems,
is constructed for the subsystems to asymptotically estimate the exogenous signal.
The case with uncertain linear subsystems is further studied.
The proposed distributed adaptive observer and the internal model principle are combined to design
distributed controllers to solve the robust cooperative output regulation problem.
The proposed control schemes in this paper, in contrast to the controllers in \cite{su2012cooperative,su2012cooperative,su2013general,yu2013robust,wang2010distributed,hong2013distributed},
can be designed and implemented by each subsystem in a fully distributed fashion,
and, different from those in \cite{li2013distributeds}, are applicable to general directed graphs.

In the last part of this paper,
a special case with undirected graphs is further discussed.
A distributed adaptive output feedback control law is presented for uncertain linear subsystems.
The output feedback controller has the advantage of demanding less communication cost.
The assumptions are investigated for the existence of the distributed controllers.
A simulation example is finally presented to illustrate the effectiveness of the obtained results.



\section{Cooperative Output Regulation of Linear Multi-Agent Systems}

\subsection{Problem Statement}

In this section, we consider a network consisting of $N$ heterogeneous subsystems and an exosystem.
The dynamics of the $i$-th subsystem are described by
\begin{equation}\label{1c}
\begin{aligned}
    \dot{x}_i &=A_ix_i+B_iu_i+E_iv,\\
    e_i &=C_ix_i+D_iv,
\quad i=1,\cdots,N,
\end{aligned}
\end{equation}
where $x_i\in\mathbf{R}^{n_i}$, $u_i\in\mathbf{R}^{m_i}$, and $e_i\in\mathbf{R}^{p_i}$ are, respectively, the state,
the control input, and the regulated output of the $i$-th subsystem,
and $A_i$, $B_i$, $C_i$, and $D_i$ are constant matrices with appropriate dimensions.

In \dref{1c}, $v\in\mathbf{R}^{q}$ represents the exogenous signal
which can be either a reference input to be tracked or the disturbance to be
rejected. The exogenous signal $v$ is generated by the following exosystem:
\begin{equation}\label{1e}
\begin{aligned}
    \dot{v} &=S v,\\
    y_v &= Fv,
\end{aligned}
\end{equation}
where $y_v\in\mathbf{R}^{l}$ is the output of the exosystem,
$S\in\mathbf{R}^{q\times q}$, and $F\in\mathbf{R}^{l\times q}$.


To achieve cooperative output regulation,
the subsystems need information from other subsystems
or the exosystem. The information flow among the $N$ subsystems can be
modeled by a directed graph $\mathcal {G}=(\mathcal {V}, \mathcal
{E})$, where $\mathcal {V}=\{v_1,\cdots,v_N\}$ is the node set
 and $\mathcal {E}\subseteq\mathcal {V}\times\mathcal
{V}$ is the edge set, in which an edge is represented by an
ordered pair of distinct nodes. If $(v_i,v_j)\in\mathcal {E}$, node $v_i$
is called a
neighbor of node $v_j$. A graph is said to be undirected if
$(v_i,v_j)\in\mathcal {E}$ implies $(v_j, v_i)\in\mathcal {E}$ for
any $v_i,v_j\in\mathcal {V}$. A directed path from
node $v_{i_1}$ to node $v_{i_l}$ is a sequence of adjacent edges of
the form $(v_{i_k}, v_{i_{k+1}})$, $k=1,\cdots,l-1$.
A directed graph contains a directed spanning tree if
there exists a root node that has directed paths to all other nodes.

Since the exosystem \dref{1e} does not receive information
from any subsystem, it can be viewed as a virtual leader,
indexed by 0.
The $N$ subsystems in \dref{1c} are the followers, indexed by $1,\cdots,N$.
It is assumed that the output $y_v$ of the exosystem \dref{1e}
is available to only a subset of the followers.
Without loss of generality, suppose that
the subsystems indexed by $1,\cdots,M$ ($1\leq M \ll N$), have direct access to
the exosystem \dref{1e} and the rest of the followers do not.
The followers indexed by $1,\cdots,M$, are called the informed followers
and the rest are the uninformed ones.
The communication graph $\mathcal {G}$ among the $N$ subsystems
is assumed to satisfy the following assumption.

{\it Assumption 1:}
For each uninformed follower,
there exists at least one informed follower that has a directed path to that uninformed follower.

For the case with only one informed follower,
Assumption 1 is equivalent to that
the graph $\mathcal {G}$ contains a directed spanning tree with the informed follower as the root node.

For the directed graph $\mathcal {G}$, its adjacency matrix $\mathcal {A}=[a_{ij}]\in\mathbf{R}^{N\times
N}$ is defined by $a_{ii}=0$, $a_{ij}=1$ if $(v_j,v_i)\in\mathcal {E}$ and $a_{ij}=0$
otherwise. The Laplacian matrix $\mathcal {L}=[\mathcal
{L}_{ij}]\in\mathbf{R}^{N\times N}$ associated with $\mathcal {G}$ is defined as $\mathcal
{L}_{ii}=\sum_{j\neq i}a_{ij}$ and $\mathcal {L}_{ij}=-a_{ij}$,
$i\neq j$.

Because the informed subsystems indexed by $1,\cdots,M$, can have direct access to the exosystem \dref{1e},
they do not have to communicate with other subsystems to ensure that $e_i$, $1,\cdots,M$, converge to zero.
To avoid unnecessarily increasing the number of communication channels, we assume that the informed subsystems do not receive
information from other subsystems, i.e., they have
no neighbors except the exosystem. In this case, the Laplacian matrix $\mathcal {L}$ associated with $\mathcal {G}$
can be partitioned as
\begin{equation}\label{lapc}
\mathcal {L}=\begin{bmatrix} 0_{M\times M} & 0_{M\times (N-M)} \\
\mathcal {L}_2 & \mathcal {L}_1\end{bmatrix},
\end{equation}
where $\mathcal {L}_2\in\mathbf{R}^{(N-M)\times M}$ and $\mathcal
{L}_1\in\mathbf{R}^{(N-M)\times (N-M)}$. Under Assumption 1, it is known that
all the eigenvalues of
$\mathcal {L}_1$ have positive real parts \cite{cao2012distributed}.
Moreover, it is easy to verify that $\mathcal {L}_1$ is a nonsingular $M$-matrix \cite{qu2009cooperative},
for which we have the following result.

\begin{lemma}[\cite{qu2009cooperative,li2014TAC}]\label{ch4lemdir1}
There exists a positive diagonal matrix $G$ such that $G\mathcal {L}_1+\mathcal {L}_1^TG>0$.
One such $G$ is given by ${\mathrm{diag}}(q_{M+1},\cdots,q_{N})$, where $q =[q_{M+1},\cdots,q_{N}]^T=(\mathcal {L}_1^T)^{-1}{\bf 1}$.
\end{lemma}


The objective of the cooperative output regulation problem considered in this section
is to design appropriate distributed controllers based on the local information available to the subsystems
such that (i) The overall closed-loop system is asymptotically stable when $v=0$;
(ii) For any initial conditions $x_i(0)$, $i = 1,\cdots,N$, and
$v(0)$, $\lim_{t\rightarrow\infty} e_i(t) = 0$.

\begin{remark}
By letting $D_i=-F$ in \dref{1c} and regarding $C_ix_i$ as the output
of the $i$-th subsystem, the regulated output $e_i$ is equal to $C_ix_i-Fv$.
In this case, the cooperative output regulation problem turns out to be the leader-follower
output consensus problem as studied in \cite{wieland2011internal,yang2013output}. 
\end{remark}

To solve the above cooperative output regulation problem,
the following assumptions are needed.

{\it Assumption 2:} The matrix $S$ has no eigenvalues with negative real parts.

{\it Assumption 3:} The pairs $(A_i, B_i)$, $i=1,\cdots,N$, are stabilizable.

{\it Assumption 4:} The pair $(S, F)$ is detectable.

{\it Assumption 5:} For all $\lambda\in\sigma(S)$, where $\sigma(S)$ denotes the spectrum of $S$,
$\mathrm{rank}\left(\left[\begin{smallmatrix}A_i-\lambda I & B_i\\ C_i & 0\end{smallmatrix}\right]\right)=n_i+p_i$.

\begin{remark}
Assumptions 2--5 are the standard ones required to solve the output regulation problem
of a single linear system \cite{huang2004nonlinear}. Assumption 2 is made only for convenience.
The components of the exogenous signal $v$ corresponding to the stable eigenvalues of $S$
exponentially decay to zero and thereby will not affect the
asymptotic behavior of the closed-loop system.
\end{remark}

\subsection{Distributed Adaptive Controller Design}

Since the exogenous signal $v$ is not available to the subsystems for feedback control, the subsystems
need to implement some observers to estimate $v$.
For the informed subsystems that have direct access to the output $y_v$ of the exosystem \dref{1e},
they can estimate $v$ by using the following observers:
\begin{equation}\label{1obs0}
\begin{aligned}
    \dot{\xi}_i &=S \xi_i+L(F\xi_i- y_v),\quad i=1,\cdots,M,
\end{aligned}
\end{equation}
where the feedback gain matrix $L\in\mathbf{R}^{p\times l}$ is chosen such that
$S+LF$ is Hurwtiz. Denote by $\bar{\xi}_i=\xi_i-v$ the estimation errors.
From \dref{1e} and \dref{1obs0}, it is easy to see that
\begin{equation}\label{1obs01}
\dot{\bar{\xi}}_i =(S+LF) \bar{\xi}_i,\quad i=1,\cdots,M,
\end{equation}
implying that $\lim_{t\rightarrow \infty}\bar{\xi}_i(t)=0$, $i=1,\cdots,M.$

For the uninformed subsystems that do not have direct access to \dref{1e},
we need to construct distributed observers to estimate the exogenous signal $v$.
The distributed adaptive observer for each uninformed subsystem is described by
\begin{equation}\label{1obs}
\begin{aligned}
    \dot{\xi}_i &=S \xi_i-d_i\rho_i \sum_{j=1}^N a_{ij}(\xi_i-\xi_j),\\
    d_i &=[\sum_{j=1}^N a_{ij}(\xi_i-\xi_j)]^T\Gamma[\sum_{j=1}^N a_{ij}(\xi_i-\xi_j)], i=M+1,\cdots,N,
\end{aligned}
\end{equation}
where $\xi_i\in\mathbf{R}^{p}$, $i=M+1,\cdots,N,$ denotes
the estimate of $v$ on the $i$-th uninformed subsystem,
$d_i(t)$ denotes the time-varying coupling gain associated with the $i$-th uninformed subsystem
with $d_i(0)\geq1$,
$a_{ij}$ is the $(i,j)$-th entry of the adjacency matrix associated with
$\mathcal {G}$, $\Gamma\in\mathbf{R}^{l\times l}$ is the feedback gain matrix, and
$\rho_i(\cdot)$ are smooth
and monotonically increasing functions in terms of $\sum_{j=1}^N a_{ij}(\xi_i-\xi_j)$.
The parameters $\Gamma$ and $\rho_i(\cdot)$ are to be determined.

Since $d_i(0)\geq1$, it follows from the second equation in \dref{1obs} that $d_i(t)\geq1$ for any $t>0$.
By further noting that $\rho_i$ are monotonically increasing functions, the following lemma holds.

\begin{lemma}[\cite{li2014TAC}]\label{li2014}
For any constants $a,b>0$ and any function $\varepsilon(t)>0$, 
$$\dot{d}_i b \int_0^\varepsilon\rho_i(s)ds \leq \dot{d}_i(\frac{b^3}{3a^2}+\frac{2}{3}a\rho_i^{\frac{3}{2}}\varepsilon^{\frac{3}{2}}).$$
\end{lemma}

The following theorem designs the observers \dref{1obs0} and \dref{1obs}.

\begin{theorem}
Suppose that Assumptions 1 and 4 hold.
Then, $\lim_{t\rightarrow \infty}(\xi_i(t)- v(t))=0$, $i=1,\cdots,N$,
if $L$ in \dref{1obs0} is chosen such that
$S+LF$ is Hurwtiz and the parameters in the adaptive observer \dref{1obs}
is chosen to be $\Gamma=P^2$ and $\rho_i=(1+\zeta_i^TP\zeta_i)^3$, $i=M+1,\cdots,N$,
where $\zeta_i=\sum_{j=1}^Na_{ij}(\xi_i-\xi_j)$ and $P>0$ is a solution to the following algebraic Riccati equation (ARE):
\begin{equation}\label{alg1}
S^TP+PS+I-P^2=0.
\end{equation}
Moreover, the coupling gains $d_{i}$ in \dref{1obs} converge to some finite steady-state values.
\end{theorem}
\begin{proof}
Let $\zeta=[\zeta_{M+1}^T,\cdots,\zeta_N^T]^T$. Then,
$\zeta$ can be rewritten as
\begin{equation}\label{conerr}
\begin{aligned}
\zeta &=(\mathcal {L}_2\otimes I)\begin{bmatrix} \xi_1\\\vdots\\\xi_M\end{bmatrix}+(\mathcal {L}_1\otimes I)\begin{bmatrix} \xi_{M+1}\\\vdots\\\xi_N\end{bmatrix}\\
&=(\mathcal {L}_2\otimes I)\begin{bmatrix} \bar{\xi}_1\\\vdots\\\bar{\xi}_M\end{bmatrix}+(\mathcal {L}_1\otimes I)\begin{bmatrix} \bar{\xi}_{M+1}\\\vdots\\\bar{\xi}_N\end{bmatrix},
\end{aligned}
\end{equation}
where $\mathcal {L}_1$ and $\mathcal {L}_2$ are defined as in \dref{lapc}, and $\bar{\xi}_i=\xi_i-v$ denote the estimation errors. Because $\mathcal {L}_1$
is nonsingular and $\lim_{t\rightarrow \infty}\bar{\xi}_i(t)=0$, $i=1,\cdots,M$, it can be observed from \dref{conerr}
that $\lim_{t\rightarrow\infty} \bar{\xi}_i(t)=0$, $i=M+1,\cdots,N$, if and only if
$\lim_{t\rightarrow\infty} \zeta(t)=0$.
From \dref{1obs} and \dref{conerr}, it is not difficult to get that
$\zeta$ and $d_i$ satisfy the following dynamics:
\begin{equation}\label{netss1}
\begin{aligned}
\dot{\zeta}
&= [I_{N-M}\otimes S-\mathcal {L}_1\widehat{D}\hat{\rho} \otimes I]\zeta+(\mathcal {L}_2 \otimes LF)\bar{\xi},\\
\dot{d}_i &=\zeta_i^T\Gamma\zeta_i,
\end{aligned}
\end{equation}
where $\hat{\rho}={\rm{diag}}(\rho_{M+1},\cdots,\rho_N)$,
$\widehat{D}={\rm{diag}}(d_{M+1},\cdots,d_N)$, and $\bar{\xi}=[\bar{\xi}_1^T,\cdots,\bar{\xi}_M^T]^T$.

Let
\begin{equation}\label{lya1}
V_{1}=\sum_{i=M+1}^{N}\frac{d_iq_i}{2}\int_0^{\zeta_i^TP\zeta_i}\rho_i(s)ds+\frac{\hat{\lambda}_0}{48}\sum_{i=M+1}^{N}\tilde{d}_i^2,
\end{equation}
where $G\triangleq \mathrm{diag}(q_{M+1},\cdots,q_N)>0$ is defined as in Lemma \ref{ch4lemdir1}, $\hat{\lambda}_0$ denotes
the smallest eigenvalue of $G\mathcal {L}_1+\mathcal {L}_1^TG$, and $\tilde{d}_i\triangleq d_{i}-\alpha$, where $\alpha$ is a positive constant to be determined later.

The time derivative of $V_1$
along the trajectory of \dref{netss1} is given by
\begin{equation}\label{lyas2}
\begin{aligned}
\dot{V}_1 &=
\sum_{i=M+1}^N d_iq_i \rho_i\zeta_i^TP\dot{\zeta}_i+\sum_{i=M+1}^N\frac{\dot{d}_iq_i}{2}\int_0^{\zeta_i^TP\zeta_i}\rho_i(s)ds\\
&\quad+\frac{\hat{\lambda}_0}{24}\sum_{i=M+1}^N(d_i-\alpha)\zeta_i^TP^2\zeta_i.
\end{aligned}
\end{equation}
Note that
\begin{equation}\label{lyas3}
\begin{aligned}
&\sum_{i=M+1}^N d_i q_i \rho_i\zeta_i^TP\dot{\zeta}_i=\zeta^T(\widehat{D}\hat{\rho} G \otimes P)\dot{\zeta}\\
&\qquad\leq\frac{1}{2}\zeta^T[\widehat{D}\hat{\rho} G \otimes (PS+S^TP)-
\hat{\lambda}_0\widehat{D}^2\hat{\rho}^2\otimes P^2]\zeta\\&\quad\qquad
+\zeta^T(\widehat{D}\hat{\rho} G\mathcal {L}_2 \otimes PLF)\bar{\xi},
\end{aligned}
\end{equation}
where we have used the fact that $G\mathcal {L}_1+\mathcal {L}_1^TG\geq\hat{\lambda}_0 I$.
By using the Young's inequality \cite{bernstein2009matrix}, we can obtain that
\begin{equation}\label{lyas31}
\begin{aligned}
&\zeta^T(\widehat{D}\hat{\rho} G\mathcal {L}_2 \otimes PLF)\bar{\xi}\\&\qquad\leq \frac{\hat{\lambda}_0}{24}
\|(\widehat{D}\hat{\rho}\otimes P)\zeta\|^2+\frac{6}{\hat{\lambda}_0} \|(G\mathcal {L}_2 \otimes LF)\bar{\xi}\|^2\\
&\qquad\leq \frac{\hat{\lambda}_0}{24}
\|(\widehat{D}\hat{\rho}\otimes P)\zeta\|^2+\frac{6}{\hat{\lambda}_0} \|G\mathcal {L}_2 \otimes LF\|^2\|\bar{\xi}\|^2.
\end{aligned}
\end{equation}
In light of Lemma \ref{li2014}, we can get that 
\begin{equation}\label{lyas4}
\begin{aligned}
&\sum_{i=M+1}^N \dot{d}_i q_i \int_0^{\zeta_i^TP\zeta_i}\rho_i(s)ds \leq \sum_{i=M+1}^N(\frac{q_i^3}{3\hat{\lambda}_0^2}+\frac{2}{3}\hat{\lambda}_0\rho_i^2)\zeta_iP^2\zeta_i.
\end{aligned}
\end{equation}
Substituting \dref{lyas3}, \dref{lyas31}, and \dref{lyas4} gives
\begin{equation}\label{lyas5}
\begin{aligned}
\dot{V}_1 &\leq
\frac{1}{2}\zeta^T[\widehat{D}\hat{\rho} G \otimes (PS+S^TP)]\zeta -\sum_{i=M+1}^N
[\hat{\lambda}_0(\frac{1}{2}d_i^2\rho_i^2\\
&\quad-\frac{1}{24}d_i^2\rho_i^2-\frac{1}{24}d_i-\frac{1}{3}\rho_i^2)
+\frac{1}{24}(\hat{\lambda}_0\alpha-
\frac{4 q_i^3}{\hat{\lambda}_0^2})]\zeta^T_iP^2\zeta_i\\
&\quad+\frac{6}{\hat{\lambda}_0} \|G\mathcal {L}_2 \otimes LF\|^2\|\bar{\xi}\|^2\\
&\leq
\frac{1}{2}\zeta^T[\widehat{D}\hat{\rho} G \otimes (PS+S^TP)]\zeta\\
&\quad-
\frac{\hat{\lambda}_0}{12}\sum_{i=M+1}^N (d_i^2\rho_i^2+\hat{\alpha})\zeta^T_iP^2\zeta_i+\frac{6}{\hat{\lambda}_0} \|G\mathcal {L}_2 \otimes LF\|^2\|\bar{\xi}\|^2\\
&\leq
\frac{1}{2}\zeta^T[\widehat{D}\hat{\rho} G \otimes (PS+S^TP-P^2)]\zeta\\
&\quad+\frac{6}{\hat{\lambda}_0} \|G\mathcal {L}_2 \otimes LF\|^2\|\bar{\xi}\|^2,
\end{aligned}
\end{equation}
where we have chosen $\alpha\geq\max \,\frac{4 q_i^3}{\hat{\lambda}_0^3}+2\hat{\alpha}$ and $\sqrt{\hat{\alpha}}\geq\frac{3}{\hat{\lambda}_0}\max\,q_i$ to get the last two inequalities.

Let
\begin{equation}\label{lya1o}
V_2=\bar{\xi}^T(I_M\otimes \bar{Q})\bar{\xi},
\end{equation}
where $\bar{Q}>0$ satisfy that $\bar{Q}(S+LF)+(S+LF)^T\bar{Q}=-I$.
The time derivative of $V_3$ along \dref{1obs01}
can be obtained as
\begin{equation}\label{1lay3o}
\begin{aligned}
    \dot{V}_2
    &=\bar{\xi}^T(I_M\otimes \bar{Q})\dot{\bar{\xi}}=-\|\bar{\xi}\|^2.
\end{aligned}
\end{equation}

Consider the following Lyapunov function candidate:
$$V=V_1+h_2V_2,$$
where $h_2\geq\frac{6}{\hat{\lambda}_0} \|G\mathcal {L}_2 \otimes LF\|^2$.
Since $d_i(t)\geq1$ for any $t>0$ and
$\rho_i(\cdot)$ are monotonically increasing functions
satisfying $\rho_i(s)\geq1$ for $s>0$, it is not difficult to see that $V$ is positive definite with respect to $\bar{\xi}$, $\zeta_i$,
and $\tilde{d}_i$, $i=M+1,\cdots,N$.
By using \dref{lyas5} and \dref{1lay3o}, we can get the
time derivative of $V$ as
\begin{equation}\label{lya7}
\begin{aligned}
\dot{V_{1}}&\leq -\frac{1}{2}\zeta^T[\widehat{D}\hat{\rho}G\otimes I]\zeta \\
&\leq -\frac{1}{2}\min\, q_i\|\zeta\|^2\leq0.
\end{aligned}
\end{equation}

From \dref{lya7}, we can get that each $d_i$ is bounded, which, by noting $\dot{d}_i\geq 0$, implies that each
$d_i$ converges to some finite value.
Note that $\dot{V}_1\equiv0$ is equivalent to $\zeta\equiv0$. By LaSalle's Invariance principle \cite{khalil2002nonlinear}, it follows that $\zeta$ asymptotically converges to zero. That is, $\lim_{t\rightarrow \infty}(\xi_i(t)- v(t))=0$, $i=M+1,\cdots,N$.
\end{proof}

\begin{remark}
Theorem 1 shows that the local observer \dref{1obs} and the distributed adaptive observer
\dref{1obs0} ensure that the subsystems can asymptotically
estimate the exogenous signal for general directed graphs
satisfying Assumption 1.
Because $(S,I)$ is controllable, the ARE \dref{alg1} has a unique solution $P>0$. That is, the adaptive observer \dref{1obs} always exists.
\end{remark}

Upon the basis of the estimates $\xi_i$ of the exogenous signal $v$, we propose the following controller to each subsystem as
\begin{equation}\label{1con}
u_i =K_{1i} x_i+K_{2i} \xi_i, \quad i=1,\cdots,N,
\end{equation}
where $K_{1i}\in\mathbf{R}^{m_i\times n_i}$ and $K_{2i}\in\mathbf{R}^{m_i\times q}$ are the feedback gain matrices.

By substituting \dref{1con} into \dref{1c}, we write the closed-loop dynamics of the subsystems as
\begin{equation}\label{1closed}
\begin{aligned}
    \dot{x}_i &=(A_i+B_iK_{1i})x_i +E_iv+B_iK_{2i}\xi_i,\\
    e_i &=C_ix_i+D_iv,
\quad i=1,\cdots,N.
\end{aligned}
\end{equation}

\begin{theorem}
Suppose that Assumptions 1--5 hold. Select $K_{1i}$ such that $A_i+B_iK_{1i}$ are Hurwitz,
and $K_{2i}=U_i-K_{1i}X_i$, $i=1,\cdots,N$, where $(X_i,U_i)$ are solutions to the regulator equations:
\begin{equation}\label{1regulator}
\begin{aligned}
    X_iS &=A_iX_i+B_iU_i +E_i,\\
    0 &=C_iX_i+D_i, \quad i=1,\cdots,N.
\end{aligned}
\end{equation}
Then, the cooperative output regulation problem is solved by
the distributed controller \dref{1con} and the adaptive observers
\dref{1obs0} and \dref{1obs} constructed by Theorem 1.
\end{theorem}
\begin{proof}
The closed-loop dynamics of each subsystem can be rewritten as
\begin{equation}\label{1closed2}
\begin{aligned}
    \dot{x}_i &=(A_i+B_iK_{1i})x_i +(E_i+B_iK_{2i})v+B_iK_{2i}\bar{\xi}_i,\\
    e_i &=C_ix_i+D_iv,
\quad i=1,\cdots,N,
\end{aligned}
\end{equation}
where $\bar{\xi}_i=\xi_i-v$ denote the estimation errors.
Since $A_i+B_iK_{1i}$ are Hurwitz and $\lim_{t\rightarrow \infty}(\xi_i(t)- v(t))=0$, $i=1,\cdots,N$,
it is easy to see that $x_i$, $i=1,\cdots,N$, asymptotically converge to zero in the case of $v=0$.

Let $\tilde{x}_i=x_i-X_iv$, $i=1,\cdots,N.$ Then, by invoking \dref{1regulator},
we can obtain from \dref{1closed2} and \dref{1e} that
\begin{equation}\label{1closed30}
\begin{aligned}
    \dot{\tilde{x}}_i &=(A_i+B_iK_{1i})\tilde{x}_i +B_iK_{2i}\bar{\xi}_i,\\
    e_i &=C_i\tilde{x}_i,\quad i=1,\cdots,N.
\end{aligned}
\end{equation}
Let
$$V_3=\sum_{i=1}^N\tilde{x}_i^TQ_i\tilde{x}_i,$$
where $Q_i>0$ satisfy that $Q_i(A_i+B_iK_{1i})+(A_i+B_iK_{1i})^TQ_i=-2I$.
The time derivative of $V_3$ along \dref{1closed30} can be obtained as
\begin{equation}\label{1lay2}
\begin{aligned}
    \dot{V}_3
    &=-4\sum_{i=1}^N\|\tilde{x}_i\|^2+2\sum_{i=1}^N\tilde{x}_i^TQ_iB_iK_{2i}\bar{\xi}_i\\
    &\leq -3\|\tilde{x}\|^2+\underset{i=1,\cdots,N}{\max}\|Q_iB_iK_{2i}\|^2\sum_{i=1}^N\|\bar{\xi}_i\|^2,
\end{aligned}
\end{equation}
where $\tilde{x}=[\tilde{x}_1^T,\cdots,\tilde{x}_N^T]^T$.
From \dref{conerr},
we can obtain that
\begin{equation}\label{1lay20}
\begin{aligned}
    \sum_{i=M+1}^N\|\bar{\xi}_i\|^2\leq2\|\mathcal {L}_1^{-1}\otimes I\|^2\|\zeta\|^2+2\|\mathcal {L}_1^{-1}\mathcal {L}_2\otimes I\|^2\sum_{i=1}^M\|\bar{\xi}_i\|^2.
\end{aligned}
\end{equation}
Substituting \dref{1lay20} into \dref{1lay2} yields
\begin{equation}\label{1lay21}
\begin{aligned}
    \dot{V}_3
    &\leq -3\|\tilde{x}\|^2+\epsilon_1\|\zeta\|^2+\epsilon_2\sum_{i=1}^M\|\bar{\xi}_i\|^2,
\end{aligned}
\end{equation}
where
$$
\begin{aligned}
\epsilon_1 &=2\underset{i=1,\cdots,N}{\max}\|Q_iB_iK_{2i}\|^2\|\mathcal {L}_1^{-1}\otimes I\|^2,\\
\epsilon_2 &=\underset{i=1,\cdots,N}{\max}\|Q_iB_iK_{2i}\|^2+2\|\mathcal {L}_1^{-1}\mathcal {L}_2\otimes I\|^2.
\end{aligned}
$$

Consider the following Lyapunov function candidate:
$$\widetilde{V}=\tilde{h}_1V_1+\tilde{h}_2V_2+V_3,$$
where $V_1$ and $V_2$ are defined as in \dref{lya1} and \dref{lya1o},
$\tilde{h}_1=\frac{2\epsilon_1}{\min\, q_i}$, $\tilde{h}_2= \frac{6\tilde{h}_1}{\hat{\lambda}_0}\|G\mathcal {L}_2 \otimes LF\|^2+\epsilon_2.$
By using \dref{lya7}, \dref{1lay21}, and \dref{1lay3o}, we can get the
time derivative of $V$ as
$$    \dot{\widetilde{V}}
    \leq-3\|\tilde{x}\|^2.$$
 By LaSalle's Invariance principle \cite{khalil2002nonlinear}, it follows that $\lim_{t\rightarrow\infty} \tilde{x}(t) = 0$, which,
 in light of the second equation in \dref{1closed30}, implies that $\lim_{t\rightarrow\infty} e_i(t) = 0$, $i=+1,\cdots,N$. That is,
 the cooperative output regulation problem is solved.
\end{proof}

\begin{remark}
According to Theorem 1.9 in \cite{huang2004nonlinear}, each regulator equation in \dref{1regulator} has a unique solution $(X_i,U_i)$
if and only if Assumption 5 holds.
\end{remark}

\begin{remark}
Theorem 2 states that the proposed adaptive control scheme consisting of the controller \dref{1con} and the observers \dref{1obs} and \dref{1obs0}
can solve the cooperative output regulation problem.
Note that the design of the proposed control scheme relies on the subsystem dynamics and the local information of neighboring subsystems,
independent of any global information of the communication graph. Therefore, the proposed control scheme in this section is fully distributed.
By comparison, the controllers in the previous work \cite{su2012cooperative} require some nonzero eigenvalue of the Laplacian matrix which is global information of the communication graph. 
The adaptive controllers in \cite{li2013distributeds} are indeed fully distributed, which, however, are applicable to only undirected graphs. The proposed control scheme in this section works for general directed graphs, whose design is more challenging.
\end{remark}

\section{Robust Cooperative Output Regulation of Linear Uncertain Multi-Agent Systems}

\subsection{Problem Formulation}
In this section, we consider the case where the subsystems in \dref{1c} are subject to uncertainties
and have the same dimensions that are chosen to
be $x_i\in\mathbf{R}^{n}$, $u_i\in\mathbf{R}^{m}$, and $e_i\in\mathbf{R}^{p}$. Specifically,
the matrices in \dref{1c} can be written as
\begin{equation}\label{1cun0}
\begin{aligned}
    A_i=\bar{A}_i+\Delta A_i, ~B_i=\bar{B}_i+\Delta B_i,~E_i=\bar{E}_i+\Delta E_i\\
    C_i=\bar{C}_i+\Delta C_i, ~D_i=\bar{D}_i+\Delta D_i, \quad i=1,\cdots,N,
\end{aligned}
\end{equation}
where $\bar{A}_i$, $\bar{B}_i$, $\bar{E}_i$, $\bar{C}_i$, $\bar{D}_i$ denote the nominal parts of these matrices,
and $\Delta A_i$, $\Delta B_i$, $\Delta E_i$, $\Delta C_i$, $\Delta D_i$ are the uncertainties associated with these matrices.
For convenience, let $\hat{\Delta}$ represents the uncertainty vector, defined by
$\hat{\Delta}=\left[\begin{smallmatrix} \mathrm{vec}(\Delta A_1,\cdots,\Delta A_N)\\\mathrm{vec}(\Delta B_1,\cdots,\Delta B_N)\\\mathrm{vec}(\Delta E_1,\cdots,\Delta E_N)\\\mathrm{vec}(\Delta C_1,\cdots,\Delta C_N)\\\mathrm{vec}(\Delta D_1,\cdots,\Delta D_N)\end{smallmatrix}\right],$
where $\mathrm{vec}(X)$ is a column vector formed by all the columns of matrix $X$.

The communication graph $\mathcal {G}$ among the $N$ uncertain subsystems is directed and satisfies Assumption 1.
The exosystem is described by \dref{1e}. For the uncertain subsystems described by \dref{1c} and \dref{1cun0} and the exosystem \dref{1e},
the robust cooperative output regulation problem in this section
is to design appropriate distributed controllers based on the local information
such that (i) The overall nominal closed-loop system with $\hat{\Delta}=0$ is asymptotically stable when $v=0$;
(ii) there exists an open neighborhood $W$ of the origin, for any
$\hat{\Delta}\in W$ and any initial condition $x_i(0)$, $i = 1,\cdots,N$, and
$v(0)$, $\lim_{t\rightarrow\infty} e_i(t) = 0$.

\subsection{Distributed Adaptive State Controller}

The internal model principle will be utilized to solve the robust cooperative output regulation problem.
The concept of the $p$-copy internal model is introduced as follows \cite{davison1976robust,huang2004nonlinear}.

{\it Definition 1}:
A pair of matrices $(G_1,G_2)$ is said to incorporate the
$p$-copy internal model of the matrix $S$ if
\begin{equation}\label{internalmodel}
G_1=\mathrm{diag}(\underbrace{\beta,\cdots,\beta}_{p-tuple}),\quad G_2=\mathrm{diag}(\underbrace{\sigma,\cdots,\sigma}_{p-tuple}),
\end{equation}
where $\beta$ is a square matrix and $\sigma$ is a column vector such that $(\beta,\sigma)$ is
controllable and the minimal polynomial of $\beta$  equals the characteristic
polynomial of $S$.

Using the estimates $\xi_i$ of the exogenous signal $v$ via the observers \dref{1obs} and \dref{1obs0} and the above
$p$-copy internal model, we introduce the following distributed dynamic state feedback control law:
\begin{equation}\label{1cun}
\begin{aligned}
    u_i &=K_{xi}x_i+K_{zi} z_i,\\
    \dot{z}_i &=G_1z_i+G_2(C_ix_i+D_i\xi_i), \quad i=1,\cdots,N,
\end{aligned}
\end{equation}
where $z_i\in\mathbf{R}^{n_z}$ with $n_z$ to be specified later, and $K_{xi}$ and $K_{zi}$ are the feedback gain matrices
to be designed.

By combining \dref{1cun} and \dref{1c}, we get the augumented closed-loop dynamics of the subsystems as
\begin{equation}\label{1closedun}
\begin{aligned}
    \dot{\eta}_i &=A_{ci}\eta_i+B_{ci}v+Y_{ci}\bar{\xi}_i,\\
    e_i &=C_{ci}\eta_i+D_iv,
\quad i=1,\cdots,N,
\end{aligned}
\end{equation}
where $\eta_i=[x_i^T,z_i^T]^T$, the estimation errors $\bar{\xi}_i$ are defined as \dref{1closed2}, and
$$\begin{aligned}A_{ci} &=\begin{bmatrix}A_i+B_iK_{xi} & B_iK_{zi}\\G_2C_i & G_1\end{bmatrix},
~B_{ci} =\begin{bmatrix} E_i\\ G_2D_i\end{bmatrix},\\
C_{ci} &=\begin{bmatrix} C_i & 0\end{bmatrix}, ~Y_{ci}=\begin{bmatrix} 0\\ G_2D_i\end{bmatrix}.
\end{aligned}$$

\begin{theorem}
Suppose that Assumptions 1--5 hold. Choose $K_{xi}$ and $K_{zi}$ such that
$\left[\begin{smallmatrix}\bar{A}_i+\bar{B}_iK_{xi} & \bar{B}_iK_{zi}\\G_2\bar{C}_i & G_1\end{smallmatrix}\right]$ are Hurwitz, $i=1,\cdots,N$,
Then, the robust cooperative output regulation problem is solved by
the distributed controller \dref{1cun} and the adaptive observers
\dref{1obs0} and \dref{1obs} constructed by Theorem 1.
\end{theorem}
\begin{proof}
Since the nominal forms of the system matrices $A_{ci}$ of \dref{1closedun}, equal to $\left[\begin{smallmatrix}\bar{A}_i+\bar{B}_iK_{xi} & \bar{B}_iK_{zi}\\G_2\bar{C}_i & G_1\end{smallmatrix}\right]$, are Hurwitz, there exists an open neighborhood $W$ such that for any $\hat{\Delta}\in W$,
the state matrices $A_{c i}$ are also Hurwitz. Because $(G_1,G_2)$ incorporates a $p$-copy internal model of $S$, it follows from Lemma
1.27 of \cite{huang2004nonlinear} that for any $\hat{\Delta}\in W$, there exist $X_{xi}$ and $X_{zi}$ such that
\begin{equation}\label{2lya10}
\begin{aligned}
    X_{xi}S &=(A_i+B_iK_{xi})X_{xi}+B_iK_{zi}X_{zi}+E_i,\\
    X_{zi}S &=G_1X_{zi}+G_2(C_iX_{xi}+D_i),\\
    0 &=C_{i}X_{xi}+D_i,\quad i=1,\cdots,N.
\end{aligned}
\end{equation}
Let $X_{ci}=\left[\begin{smallmatrix}X_{xi}\\ X_{zi}\end{smallmatrix}\right]$. Then, \dref{2lya10} can be rewritten as
\begin{equation}\label{2lya1}
\begin{aligned}
    X_{ci}S &=A_{ci}X_{ci}+B_{ci},\\
    0 &=C_{ci}X_{ci}+D_i,\quad i=1,\cdots,N.
\end{aligned}
\end{equation}
Let $\tilde{\eta}_i=\eta_i-X_{ci}v$, $i=1,\cdots,N.$ Then,
we can obtain from \dref{1closedun}, \dref{2lya1}, and \dref{1e} that
\begin{equation}\label{1closed3}
\begin{aligned}
    \dot{\tilde{\eta}}_i &=A_{ci}\tilde{\eta}_i +Y_{ci}\bar{\xi}_i,\\
    e_i &=C_{ci}\tilde{x}_i,\quad i=1,\cdots,N.
\end{aligned}
\end{equation}
From \dref{1closed3}, we see that $\lim_{t\rightarrow\infty} e_i(t)=0$ if $\lim_{t\rightarrow\infty} \tilde{\eta}_i(t)=0$,
the latter of which can be shown by following similar steps in the proof of Theorem 2.
\end{proof}

\begin{remark}
By choosing $(G_1,G_2)$ which incorporates a $p$-copy internal model of $S$ in the specific form as in Remark 1.23 of \cite{huang2004nonlinear}, $G_1$ has the property that $\mathrm{rank}\left(\left[\begin{smallmatrix} \bar{A}_i-\lambda I & \bar{B}_i\\ \bar{C}_i & 0\end{smallmatrix}\right]\right)=n+p,$ for all $\lambda\in\sigma(G_1)$. By Lemma 1.26 in \cite{huang2004nonlinear}, under Assumption 2 and 5, the pairs $\left(\left[\begin{smallmatrix} \bar{A}_i & 0\\ G_2\bar{C}_i & G_1\end{smallmatrix}\right],\left[\begin{smallmatrix} \bar{B}_i\\ 0\end{smallmatrix}\right]\right)$ is stabilizable. Therefore,
$K_{xi}$ and $K_{zi}$ do exist such that $\left[\begin{smallmatrix}\bar{A}_i+\bar{B}_iK_{xi} & \bar{B}_iK_{zi}\\G_2\bar{C}_i & G_1\end{smallmatrix}\right]$ are Hurwitz.
\end{remark}

\begin{remark}
The robust cooperative output regulation problem is also studied in the previous works \cite{wang2010distributed,yu2013robust,su2013general}. Note that those controllers in \cite{wang2010distributed,yu2013robust,su2013general} depend on global information of
the communication graph and thereby are not fully distributed. By comparison,
one favorable feature of the proposed adaptive control scheme in this section is that
by using the adaptive observer \dref{1obs} to estimate the exogenous signal $v$,
it is fully distributed.
\end{remark}

\subsection{Distributed Adaptive Output Controller for Undirected Graphs}

It is worth noting that in the adaptive observer \dref{1obs}, each subsystem needs to transmit its own estimate $\xi_i$ of $v$ to its neighbors.
However, it is more desirable to transmit a part of the estimates $\xi_i$ or the outputs of $\xi_i$
between neighboring subsystems, which reduces the communication burden. In this subsection, we will present a novel adaptive observer
based on the outputs of $\xi_i$ of neighboring subsystems, for the special case where the communication graph satisfies the following assumption:

{\it Assumption 6:}
The communication graph $\mathcal {G}$ satisfies Assumption 1 and
the subgraph of the uninformed subsystems is undirected.

The novel adaptive observer of each uninformed subsystem is described by
\begin{equation}\label{1obso}
\begin{aligned}
    \dot{\xi}_i &=S \xi_i+c_iJ\sum_{j=1}^N a_{ij}(\mu_i-\mu_j),\\
    \mu_i &=F\xi_i,\\
    \dot{c}_i &=\tau_i[\sum_{j=1}^N a_{ij}(\mu_i-\mu_j)]^T[\sum_{j=1}^N a_{ij}(\mu_i-\mu_j)],
\end{aligned}
\end{equation}
where $\xi_i\in\mathbf{R}^{p}$, $i=M+1,\cdots,N,$ denotes
the estimate of $v$ on the $i$-th uninformed subsystem,
$\xi_i$, $i=1,\cdots,M,$ are given by \dref{1obs0},
$c_i(t)$ denotes the coupling gain associated with the $i$-th uninformed subsystem
with $c_i(0)\geq0$, $\tau_i$ are positive scalars,
$J\in\mathbf{R}^{p\times l}$ is the feedback gain matrix to be determined.
Note that the term
$\sum_{j=1}^N a_{ij}(\mu_i-\mu_j)$ in \dref{1obso} implies that
the subsystems need to transmit the virtual outputs $\mu_i$ of
their estimates $\xi_i$ to their neighbors
via the communication network $\mathcal {G}$.

\begin{theorem}
Suppose that Assumptions 6 and 4 hold.
Then, the estimation errors $\bar{\xi}_i$, $i=1,\cdots,N$, defined in \dref{1closed2} asymptotically converge to zero,
if $L$ in \dref{1obs0} is chosen such that $L=J$ and the parameter in the adaptive observer \dref{1obso}
is chosen to be $J=\tilde{P} F^T$,
where $\tilde{P}>0$ is a solution to the following ARE:
\begin{equation}\label{alg2}
\tilde{P}S^T+S\tilde{P}+I-\tilde{P}F^TF\tilde{P}=0.
\end{equation}
\end{theorem}

\begin{proof} Note that $S+LF$ with $L+J$ is Hurwitz, which follows readily from \dref{alg2}.
The convergence of the estimation errors $\bar{\xi}_i$, $i=1,\cdots,M$, to zero is obvious. In the following, we will show
the convergence of the rest estimation errors.

Under Assumption 6, it is known that all the eigenvalues of
$\mathcal {L}_1$ are positive, each entry of $-\mathcal
{L}_1^{-1}\mathcal {L}_2$ is nonnegative, and each row of
$-\mathcal {L}_1^{-1}\mathcal {L}_2$ has a sum equal to one \cite{cao2012distributed}.
Let
$$\varrho=\begin{bmatrix}\varrho_{M+1}\\\vdots\\\varrho_N\end{bmatrix}=\begin{bmatrix} \xi_{M+1}\\\vdots\\\xi_N\end{bmatrix}-(\mathcal {L}_1^{-1}\mathcal {L}_2\otimes I)\begin{bmatrix} \xi_1\\\vdots\\\xi_M\end{bmatrix}.$$
It is easy to see that $\lim_{t\rightarrow \infty}\bar{\xi}_i(t)=0$, $i=M+1,\cdots,N$, if $\lim_{t\rightarrow \infty}\varrho(t)=0$.
The system \dref{1obso} can be rewritten in terms of $\varrho$ as
\begin{equation}\label{1obso2}
\begin{aligned}
    \dot{\varrho} &=(I_{N-M}\otimes S +\widehat{C}\mathcal {L}_1\otimes JF)\varrho+(\mathcal {L}_2\otimes LF)\bar{\xi},\\
    \dot{c}_i &=\tau_i[\sum_{j={M+1}}^N \mathcal {L}_{ij}F\varrho_i]^T[\sum_{j={M+1}}^N \mathcal {L}_{ij}F\varrho_i],
\end{aligned}
\end{equation}
where $\widehat{C}={\rm{diag}}(c_{M+1},\cdots,c_N)$, $\bar{\xi}$ is defined as in \dref{netss1},
 and $\mathcal {L}_{ij}$ denotes the $(i, j)$-th entry of the Laplacian
matrix $\mathcal {L}$.

Consider the Lyapunov function candidate:
\begin{equation}\label{lyan1}
V_4=\frac{1}{2}\varrho^T(\mathcal {L}_1\otimes \tilde{P}^{-1})\varrho+
\sum_{i=M+1}^{N}\frac{(c_{i}-\alpha)^2}{2\tau_{i}}+\beta V_2,
\end{equation}
where $\alpha$ is a positive constant and $V_2$ is defined as in \dref{lya1o}.
The time derivative of $V_4$ along the trajectory of \dref{1obso2} can be obtained as
\begin{equation}\label{1obso2}
\begin{aligned}
    \dot{V}_4 &=\varrho^T(\mathcal {L}_1\otimes \tilde{P}^{-1}S +\mathcal {L}_1\widehat{C}\mathcal {L}_1\otimes F^TF)\varrho\\
    &\quad +\sum_{i={M+1}}^N(c_{i}-\alpha)[\sum_{j={M+1}}^N \mathcal {L}_{ij}\varrho_i]^T F^TF[\sum_{j={M+1}}^N \mathcal {L}_{ij}\varrho_i]\\
    &\quad+\varrho^T(\mathcal {L}_1\mathcal {L}_2\otimes F^TF)\bar{\xi}+\beta \dot{V}_2.
\end{aligned}
\end{equation}
Note that
\begin{equation}\label{lyan3}
\begin{aligned}
&\varrho^T(\mathcal {L}_1\widehat{C}\mathcal {L}_1\otimes F^TF)\varrho\\&\qquad=\sum_{i={M+1}}^Nc_i[\sum_{j={M+1}}^N \mathcal {L}_{ij}\varrho_i]^T F^TF[\sum_{j={M+1}}^N \mathcal {L}_{ij}\varrho_i],
\end{aligned}
\end{equation}
and
\begin{equation}\label{lyan3o}
\begin{aligned}
&\varrho^T(\mathcal {L}_1\mathcal {L}_2\otimes F^TF)\bar{\xi}\leq\frac{1}{2}\|(\mathcal {L}_1\otimes F)\varrho\|^2+\frac{1}{2}\|\mathcal {L}_2\otimes F\|^2\|\bar{\xi}\|^2.
\end{aligned}
\end{equation}
Let $\beta= \frac{1}{2}\|\mathcal {L}_2\otimes F\|^2$.
Substituting \dref{lyan3}, \dref{lyan3o}, and \dref{1lay3o} into \dref{1obso2} yields
\begin{equation}\label{lyan4}
\begin{aligned}
    \dot{V}_4 &=\frac{1}{2}\varrho^T[\mathcal {L}_1\otimes (\tilde{P}^{-1}S+S^T\tilde{P}^{-1}) -(2\alpha+1)\mathcal {L}_1^2\otimes F^TF)\varrho.
\end{aligned}
\end{equation}
Since \dref{alg2} holds and $\mathcal {L}_1>0$, we can choose $\alpha$ to be sufficiently large such that
$\mathcal {L}_1\otimes (\tilde{P}^{-1}S+S^T\tilde{P}^{-1}) -(2\alpha+1)\mathcal {L}_1^2\otimes F^TF<0$. Therefore, we get from \dref{lyan4} that
$\dot{V}_4\leq0$. By using LaSalle's Invariance principle \cite{khalil2002nonlinear}, it follows that $\lim_{t\rightarrow\infty} \varrho(t) = 0$, which implies that $\lim_{t\rightarrow \infty}\bar{\xi}_i(t)=0$, $i=M+1,\cdots,N$.
\end{proof}


By using the observers \dref{1obs0} and \dref{1obso}, we propose the following distributed dynamic output feedback control law to each subsystem as
\begin{equation}\label{1cuno}
\begin{aligned}
    u_i &=K_i z_i,\\
    \dot{z}_i &=\mathcal {P}_{1i}z_i+\mathcal {P}_{2i}(C_ix_i+D_i\xi_i), \quad i=1,\cdots,N,
\end{aligned}
\end{equation}
where 
$K_i =\left[\begin{smallmatrix} K_{xi}  & K_{zi}\end{smallmatrix}\right]$,
$\mathcal {P}_{1i} =\left[\begin{smallmatrix} \bar{A}_i+\bar{B}_iK_{xi}-L_i\bar{C}_i  & \bar{B}_iK_{zi}\\ 0 & G_{1}\end{smallmatrix}\right]$,
$\mathcal {P}_{2i}=\left[\begin{smallmatrix}L_i  \\ G_{2}\end{smallmatrix}\right]$,
$K_{xi}$  $K_{zi}$, $G_1$, and $G_2$ are defined as in \dref{1cun}, and $L_i$ needs to be determined.
By substituting \dref{1cuno} into \dref{1c}, we get the augumented closed-loop dynamics of the subsystems as
\begin{equation}\label{1closeduno}
\begin{aligned}
    \dot{\eta}_i &=\tilde{A}_{ci}\eta_i+\tilde{B}_{ci}v+\tilde{Y}_{ci}\bar{\xi}_i,\\
    e_i &=\tilde{C}_{ci}\eta_i+D_iv,
\quad i=1,\cdots,N,
\end{aligned}
\end{equation}
where
$$\begin{aligned}\tilde{A}_{ci} &=\begin{bmatrix}A_i & B_iK_{i}\\\mathcal {P}_{2i}C_i & G_1\end{bmatrix},
~\tilde{B}_{ci} =\begin{bmatrix} E_i\\ \mathcal {P}_{2i}D_i\end{bmatrix},\\
\tilde{C}_{ci} &=\begin{bmatrix} C_i & 0\end{bmatrix}, ~\tilde{Y}_{ci}=\begin{bmatrix} 0\\ \mathcal {P}_{2i}D_i\end{bmatrix}.
\end{aligned}$$

\begin{theorem}
Suppose that Assumptions 1--5 hold and that $(\bar{A}_i,\bar{C}_i$) is observable.
Choose $K_{xi}$ and $K_{zi}$ as in Theorem 3 and $L_i$ such that
$\bar{A}_i-L_i\bar{C}_i$ are Hurwitz.
Then, the robust cooperative output regulation problem is solved by
the output feedback control law \dref{1cuno} and the adaptive observers
\dref{1obs0} and \dref{1obso} constructed by Theorem 4.
\end{theorem}
\begin{proof}
The nominal forms of the system matrices $\tilde{A}_{ci}$ are equal to
$\left[\begin{smallmatrix} \bar{A}_i & \bar{B}_iK_{xi} & \bar{B}_iK_{zi}\\
L_i\bar{C}_i & \bar{A}_i +\bar{B}_iK_{xi}-L_i\bar{C}_i &\bar{B}_iK_{zi}\\
G_2\bar{C}_i & 0 & G_1\end{smallmatrix}\right].$
Multiplying the left hand side of the above matrix by $T=\left[\begin{smallmatrix}I & 0 & 0\\ 0 & 0 & I \\-I & I & 0\end{smallmatrix}\right]$ and the right hand side by $T^{-1}$ gives
$\left[\begin{smallmatrix}\bar{A}_i+\bar{B}_iK_{xi} & \bar{B}_iK_{zi} & \bar{B}_iK_{xi}\\G_2\bar{C}_i & G_1 & 0\\ 0 & 0 &
\bar{A}_i -L_i\bar{C}_i\end{smallmatrix}\right].$
Therefore, it is easy to see that the nominal forms of $\tilde{A}_{ci}$ are Hurwitz, implying that
there exists an open neighborhood $W$ such that for any $\hat{\Delta}\in W$,
the state matrices $\tilde{A}_{c i}$ are also Hurwitz.
The rest of the proof is similar to the proof of Theorem 3.
\end{proof}

\begin{remark}
Compared to the state feedback adaptive controllers in \cite{li2013distributeds},
The proposed control scheme in this section is based on the local output information,
which requires less communication burden.
\end{remark}

\section{Numerical Simulation}

In this section, a simulation example will be presented for illustration.

The dynamics of the subsystems are described by \dref{1c}, with
$$\begin{aligned}
A_i &=\begin{bmatrix} 0 & 1\\ \delta_{i1} & \delta_{i2} \end{bmatrix}, ~B_i=\begin{bmatrix} 0 \\ 2 \end{bmatrix},~
E_i=\begin{bmatrix} \varsigma_i & 0\\ 0 & 1 \end{bmatrix},\\
C_i &=\begin{bmatrix} 1 &  0 \end{bmatrix},
~D_i=\begin{bmatrix} 0 &  2 \end{bmatrix},
\end{aligned}$$
where $\varsigma_i$ are randomly chosen within the interval $[1,3]$
and $\Delta A_i =\left[\begin{smallmatrix} 0 & 0\\ \delta_{i1} & \delta_{i2} \end{smallmatrix}\right]$,
with $\delta_{i1}$ and $\delta_{i2}$ randomly chosen within $(0,0.06]$,
denotes the uncertainty associated with $A_i$.
The exosystem is described by \dref{1e}, with
$S =\left[\begin{smallmatrix} 0 & 1\\ -2 & 0 \end{smallmatrix}\right]$
and $F=-D_i$. It is easy to verify that Assumptions 2--5 are satisfied.
The information flow among all subsystems and the exosystem is depicted as the
directed graph in Fig. 1, where the node indexed by 0 denotes the exosystem,
the node indexed by 1 is the informed follower and the rest are the uninformed followers.
Clearly, Assumption 1 holds.

\begin{figure}[htbp]
\centering
\includegraphics[width=0.6\linewidth]{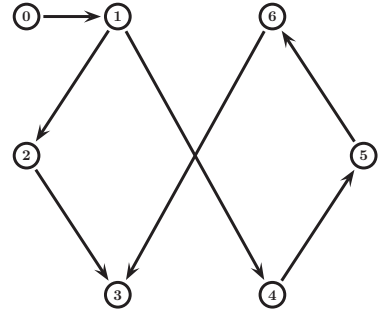}
\caption{The information flow among all subsystems and the exosystem. }
\end{figure}

To solve the robust cooperative output regulation problem, we will implement the observers \dref{1obs} and \dref{1obs0}
and the control law \dref{1cun}. As shown in Theorem 1, choose $L=\left[\begin{smallmatrix} 0 \\ 1.5 \end{smallmatrix}\right]$
such that $S+LF$ is Hurwitz. Solving the ARE \dref{alg1} gives $P=\left[\begin{smallmatrix} 1.2739  & -0.1623\\
   -0.1623  &  0.8057 \end{smallmatrix}\right]$,
implying that $\Gamma$ in \dref{1obs} equals $\left[\begin{smallmatrix} 1.6491 &  -0.3375\\
   -0.3375  &  0.6754\end{smallmatrix}\right]$.
Following Remark 1.23 in \cite{huang2004nonlinear}, let $G_1=\left[\begin{smallmatrix} 0 & 1\\ -2 & 0 \end{smallmatrix}\right]$
and $G_2=\left[\begin{smallmatrix} 0 \\ 1 \end{smallmatrix}\right]$.
Using Theorem 2, select $K_{xi}$ and $K_{zi}$ in \dref{1cun} to be
$K_{xi}=-\left[\begin{smallmatrix}4.95 &  2.85\end{smallmatrix}\right]$
and $K_{zi}=\left[\begin{smallmatrix}8.1 &  0.3\end{smallmatrix}\right]$.
The simulation result is shown in Fig. 2, from which we can observe that
all regulated outputs $e_i$ of the subsystems asymptotically converge to zero.

\begin{figure}[htbp]
\centering
\includegraphics[width=0.8\linewidth,height=0.5\linewidth]{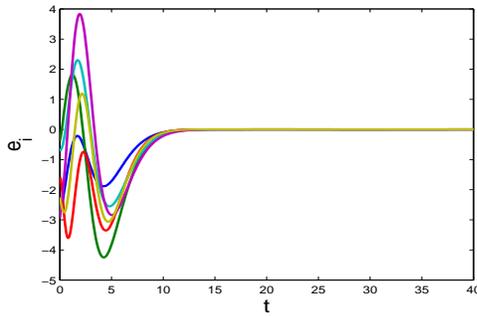}
\caption{The regulated outputs $e_i$ of the subsystems. }
\end{figure}

\section{Conclusion}

In this paper, we have presented several distributed adaptive observer-based controllers to solve the cooperative output regulation
problem for multi-agent systems with nominal or certain linear subsystems and a linear exosystem. A distinct feature of the proposed
adaptive controllers is that they can be designed and implemented by each subsystem in a fully distributed manner for general directed
graphs. This is the main contribution of this paper with respect to the existing related works.
A future research direction is to extend the idea in this paper to nonlinear multi-agent systems.



\begin{thebibliography}{10}

\bibitem{ren2007information}
W.~Ren, R.~Beard, and E.~Atkins, ``{Information consensus in multivehicle
  cooperative control},'' {\em IEEE Control Systems Magazine}, vol.~27, no.~2,
  pp.~71--82, 2007.

\bibitem{cao2013overview}
Y.~Cao, W.~Yu, W.~Ren, and G.~Chen, ``An overview of recent progress in the
  study of distributed multi-agent coordination,'' {\em IEEE Transactions on
  Industrial Informatics}, vol.~9, no.~1, pp.~427--438, 2013.

\bibitem{li2014cooperative}
Z.~Li and Z.~Duan, {\em Cooperative Control of Multi-Agent Systems: A Consensus
  Region Approach}.
\newblock Boca Raton, FL: CRC Press, 2014.

\bibitem{6573374}
H.~Su, M.~Chen, X.~Wang, and J.~Lam, ``Semiglobal observer-based
  leader-following consensus with input saturation,'' {\em IEEE Transactions on
  Industrial Electronics}, vol.~61, pp.~2842--2850, 2014.

\bibitem{6560353}
J.~Qin, C.~Yu, and H.~Gao, ``Coordination for linear multiagent systems with
  dynamic interaction topology in the leader-following framework,'' {\em IEEE
  Transactions on Industrial Electronics}, vol.~61, pp.~2412--2422, 2014.

\bibitem{5898403}
H.~Zhang, F.~Lewis, and Z.~Qu, ``Lyapunov, adaptive, and optimal design
  techniques for cooperative systems on directed communication graphs,'' {\em
  IEEE Transactions on Industrial Electronics}, vol.~59, pp.~3026--3041, 
  2012.

\bibitem{14181306}
H.~Zhang, G.~Feng, H.~Yan, and Q.~Chen, ``Observer-based output feedback
  event-triggered control for consensus of multi-agent systems,'' {\em {IEEE
  Transactions on Industrial Electronics}}, vol.~{61}, pp.~{4885--94},  
  {2014}.

\bibitem{14351657}
W.~Zeng and M.-Y. Chow, ``A reputation-based secure distributed control
  methodology in {D-NCS},'' {\em {IEEE Transactions on Industrial
  Electronics}}, vol.~{61}, pp.~{6294--303}, {2014}.

\bibitem{xiang2009synchronized}
J.~Xiang, W.~Wei, and Y.~Li, ``Synchronized output regulation of linear
  networked systems,'' {\em IEEE Transactions on Automatic Control}, vol.~54,
  no.~6, pp.~1336--1341, 2009.

\bibitem{su2012cooperative}
Y.~Su and J.~Huang, ``Cooperative output regulation of linear multi-agent
  systems,'' {\em IEEE Transactions on Automatic Control}, vol.~57, no.~4,
  pp.~1062--1066, 2012.

\bibitem{li2013distributeds}
S.~Li, G.~Feng, X.~Guan, X.~Luo, and J.~Wang, ``Distributed adaptive pinning
  control for cooperative linear output regulation of multi-agent systems,'' in
  {\em The 32nd Chinese Control Conference}, pp.~6885--6890, 2013.

\bibitem{meng2013coordinated}
Z.~Meng, T.~Yang, D.~V. Dimarogonas, and K.~H. Johansson, ``Coordinated output
  regulation of multiple heterogeneous linear systems,'' in {\em The 52nd IEEE
  Conference on Decision and Control}, pp.~2175--2180, 2013.

\bibitem{su2013general}
Y.~Su, Y.~Hong, and J.~Huang, ``A general result on the robust cooperative
  output regulation for linear uncertain multi-agent systems,'' {\em IEEE
  Transactions on Automatic Control}, vol.~58, no.~5, pp.~1275--1279, 2013.

\bibitem{yu2013robust}
L.~Yu and J.~Wang, ``Robust cooperative control for multi-agent systems via
  distributed output regulation,'' {\em Systems \& Control Letters}, vol.~62,
  no.~11, pp.~1049--1056, 2013.

\bibitem{wang2010distributed}
X.~Wang, Y.~Hong, J.~Huang, and Z.-P. Jiang, ``A distributed control approach
  to a robust output regulation problem for multi-agent linear systems,'' {\em
  IEEE Transactions on Automatic Control}, vol.~55, no.~12, pp.~2891--2895,
  2010.

\bibitem{depersisinternal}
C.~De~Persis and B.~Jayawardhana, ``On the internal model principle in the
  coordination of nonlinear systems,'' {\em IEEE Transactions on Control of
  Network Systems}, vol.~1, no.~3, pp.~272--282, 2014.

\bibitem{isidori2013robust}
A.~Isidori, L.~Marconi, and G.~Casadei, ``Robust output synchronization of a
  network of heterogeneous nonlinear agents via nonlinear regulation theory,''
  {\em IEEE Transactions on Automatic Control}, vol.~59, no.~10,
  pp.~2680--2691, 2014.

\bibitem{ding2013consensus}
Z.~Ding, ``Consensus output regulation of a class of heterogeneous nonlinear
  systems,'' {\em IEEE Transactions on Automatic Control}, vol.~58, no.~10,
  pp.~2648--2653, 2013.

\bibitem{ding2015adaptive}
Z.~Ding, ``Adaptive consensus output regulation of a class of nonlinear systems
  with unknown high-frequency gain,'' {\em Automatica}, vol.~51, pp.~348--355,
  2015.

\bibitem{li2012adaptiveauto}
Z.~Li, W.~Ren, X.~Liu, and L.~Xie, ``{Distributed consensus of linear
  multi-agent systems with adaptive dynamic protocols},'' {\em Automatica},
  vol.~49, no.~7, pp.~1986--1995, 2013.

\bibitem{li2011adaptive}
Z.~Li, W.~Ren, X.~Liu, and M.~Fu, ``{Consensus of multi-agent systems with
  general linear and Lipschitz nonlinear dynamics using distributed adaptive
  protocols},'' {\em IEEE Transactions on Automatic Control}, vol.~58, no.~7,
  pp.~1786--1791, 2013.

\bibitem{yu2013distributed}
W.~Yu, W.~Ren, W.~X. Zheng, G.~Chen, and J.~L{\"u}, ``Distributed control gains
  design for consensus in multi-agent systems with second-order nonlinear
  dynamics,'' {\em Automatica}, vol.~49, no.~7, pp.~2107--2115, 2013.

\bibitem{li2014TAC}
Z.~Li, G.~Wen, Z.~Duan, and W.~Ren, ``Designing fully distributed consensus
  protocols for linear multi-agent systems with directed graphs,'' {\em IEEE
  Transactions on Automatic Control}, in press, 2014.

\bibitem{hong2013distributed}
Y.~Hong, X.~Wang, and Z.-P. Jiang, ``Distributed output regulation of
  leader--follower multi-agent systems,'' {\em International Journal of Robust
  and Nonlinear Control}, vol.~23, no.~1, pp.~48--66, 2013.

\bibitem{cao2012distributed}
Y.~Cao, W.~Ren, and M.~Egerstedt, ``Distributed containment control with
  multiple stationary or dynamic leaders in fixed and switching directed
  networks,'' {\em Automatica}, vol.~48, no.~8, pp.~1586--1597, 2012.

\bibitem{qu2009cooperative}
Z.~Qu, {\em {Cooperative Control of Dynamical Systems: Applications to
  Autonomous Vehicles}}.
\newblock London, UK: Springer-Verlag, 2009.

\bibitem{wieland2011internal}
P.~Wieland, R.~Sepulchre, and F.~Allg{\"o}wer, ``An internal model principle is
  necessary and sufficient for linear output synchronization,'' {\em
  Automatica}, vol.~47, no.~5, pp.~1068--1074, 2011.

\bibitem{yang2013output}
T.~Yang, A.~Saberi, A.~A. Stoorvogel, and H.~F. Grip, ``Output synchronization
  for heterogeneous networks of introspective right-invertible agents,'' {\em
  International Journal of Robust and Nonlinear Control}, vol.~24, no.~13,
  pp.~1821--1844, 2014.

\bibitem{huang2004nonlinear}
J.~Huang, {\em Nonlinear Output Regulation: Theory and Applications}.
\newblock SIAM, 2004.

\bibitem{bernstein2009matrix}
D.~S. Bernstein, {\em Matrix Mathematics: Theory, Facts, and Formulas}.
\newblock Princeton University Press, 2009.

\bibitem{khalil2002nonlinear}
H.~Khalil, {\em Nonlinear Systems}.
\newblock Englewood Cliffs, NJ: Prentice Hall, 2002.

\bibitem{davison1976robust}
E.~J. Davison, ``The robust control of a servomechanism problem for linear
  time-invariant multivariable systems,'' {\em IEEE Transactions on Automatic
  Control}, vol.~21, no.~1, pp.~25--34, 1976.

\end{thebibliography}
\end{document}